\documentclass[12pt,reqno]{amsart}
\usepackage{amscd,amssymb,amsmath,color}
  \usepackage[all]{xy}
\numberwithin{equation}{section}
\usepackage{amsfonts}
\usepackage{mathtools}
\usepackage{amsthm}
\usepackage{color}

\usepackage[pagewise]{lineno}
\usepackage{braket}
\usepackage[capposition=top]{floatrow}

\usepackage{float}
\usepackage{todonotes}
\usepackage{blkarray}
\usepackage{pgf,tikz}
\usepackage{mathrsfs}
\usetikzlibrary{arrows}
\definecolor{cqcqcq}{rgb}{0.7529411764705882,0.7529411764705882,0.7529411764705882}
\definecolor{qqzzff}{rgb}{0.,0.6,1.}
\definecolor{wwccqq}{rgb}{0.4,0.8,0.}
\definecolor{ffqqqq}{rgb}{1.,0.,0.}
\usepackage{tikz}
\usetikzlibrary{shapes.geometric, arrows}
\input{prepictex.tex}
\input{pictex.tex}
\input{postpictex.tex}

\usepackage{todonotes}
\usepackage{pgf,tikz}

\title{On The Classification of Quantum Lens Spaces of Dimension at most 7}

\author{Thomas Gotfredsen}
\address{Department of Mathematics and Computer Science, University of Southern Denmark, 
Campusvej 55, 5230 Odense M, Denmark} 
\email{thgot@imada.sdu.dk} 

\author{Sophie Emma Zegers}
\address{Department of Mathematics and Computer Science, University of Southern Denmark, 
Campusvej 55, 5230 Odense M, Denmark} 
\email{mikkelsen@imada.sdu.dk}

\thanks{The work was supported by  the 
DFF-Research Project 2 on `Automorphisms and invariants of operator algebras', Nr. 7014--00145B}

\date\today



\begin{document}

\begin{abstract}
We investigate quantum lens spaces, $C(L_q^{2n+1}(r;\underline{m}))$, introduced by
Brzeziński-Szymański as graph $C^*$-algebras. 
For $n\leq 3$, we give a number-theoretic invariant, when all but one weight are coprime to the order of the acting group $r$. This builds upon the work of Eilers, Restorff, Ruiz and Sørensen.
\end{abstract}

\theoremstyle{plain}
\newtheorem{theorem}{Theorem}[section]
\newtheorem{corollary}[theorem]{Corollary}
\newtheorem{lemma}[theorem]{Lemma}
\newtheorem{proposition}[theorem]{Proposition}
\newtheorem{conjecture}[theorem]{Conjecture}
\newtheorem{commento}[theorem]{Comment}
\newtheorem{problem}[theorem]{Problem}
\newtheorem{remarks}[theorem]{Remarks}

\theoremstyle{definition}
\newtheorem{example}[theorem]{Example}

\newtheorem{definition}[theorem]{Definition}

\newtheorem{remark}[theorem]{Remark}

\newcommand{\Nb}{{\mathbb{N}}}
\newcommand{\Rb}{{\mathbb{R}}}
\newcommand{\Tb}{{\mathbb{T}}}
\newcommand{\Zb}{{\mathbb{Z}}}
\newcommand{\Cb}{{\mathbb{C}}}

\newcommand{\Af}{\mathfrak A}
\newcommand{\Bf}{\mathfrak B}
\newcommand{\Ef}{\mathfrak E}
\newcommand{\Gf}{\mathfrak G}
\newcommand{\Hf}{\mathfrak H}
\newcommand{\Kf}{\mathfrak K}
\newcommand{\Lf}{\mathfrak L}
\newcommand{\Mf}{\mathfrak M}
\newcommand{\Rf}{\mathfrak R}

\newcommand{\x}{\mathfrak x}
\def\C{\mathbb C}
\def\N{\mathbb N}
\def\R{\mathbb R}
\def\T{\mathbb T}
\def\Z{\mathbb Z}

\def\A{{\mathcal A}}
\def\B{{\mathcal B}}
\def\D{{\mathcal D}}
\def\E{{\mathcal E}}
\def\F{{\mathcal F}}
\def\G{{\mathcal G}}
\def\H{{\mathcal H}}
\def\J{{\mathcal J}}
\def\K{{\mathcal K}}
\def\LL{{\mathcal L}}
\def\N{{\mathcal N}}
\def\M{{\mathcal M}}
\def\N{{\mathcal N}}
\def\OO{{\mathcal O}}
\def\P{{\mathcal P}}
\def\Q{{\mathcal Q}}
\def\SS{{\mathcal S}}
\def\T{{\mathcal T}}
\def\U{{\mathcal U}}
\def\W{{\mathcal W}}

\def\ext{\operatorname{Ext}}
\def\span{\operatorname{span}}
\def\clsp{\overline{\operatorname{span}}}
\def\Ad{\operatorname{Ad}}
\def\ad{\operatorname{Ad}}
\def\tr{\operatorname{tr}}
\def\id{\operatorname{id}}
\def\en{\operatorname{End}}
\def\aut{\operatorname{Aut}}
\def\out{\operatorname{Out}}
\def\coker{\operatorname{coker}}
\def\pol{\mathcal{O}}

\def\la{\langle}
\def\ra{\rangle}
\def\rh{\rightharpoonup}
\def\cl{\textcolor{blue}{$\clubsuit$}}

\def\bst{\textcolor{blue}{$\bigstar$}}

\newcommand{\norm}[1]{\left\lVert#1\right\rVert}
\newcommand{\inpro}[2]{\left\langle#1,#2\right\rangle}
\newcommand{\Mod}[1]{\ (\mathrm{mod}\ #1)}

\renewcommand{\bibname}{References}
\maketitle

\addtocounter{section}{-1}

\section{Introduction}

In \cite{hs} Hong and Szymański gave a description of quantum lens spaces as graph $C^*$-algebras. This description was extended in \cite{BS} by Brzeziński and Szymański to also include weights which are not necessarily coprime with the order of the acting finite cyclic group. In noncommutative geometry quantum lens spaces are objects of increasing interest, \cite{bf, abl, dl} where noncommutative line bundles with  quantum lens spaces as total spaces are investigated. 

In the present paper we deal with classification of quantum lens spaces of dimension at most 7, with certain conditions on their weights. It can immediately be observed that two quantum lens spaces can only be isomorphic if the dimension and the order of the acting group remains the same. Classification of quantum lens spaces will thus only depend on the given weights. It is not sufficient to determine isomorphism of quantum lens spaces by only considering their K-groups and the order, \cite[Remark 7.10]{errs}.  In \cite{errs} Eilers, Restorff, Ruiz and Sørensen came with an important classification result of finite graph $C^*$-algebras using the reduced filtered K-theory. As opposed to the classification of Cuntz-Krieger algebras given by Restorff in \cite{r}, which the result in \cite{errs} is based on, quantum lens spaces fall within the scope of this classification. As an application of the classification result, Eilers, Restorff, Ruiz and Sørensen investigated 7-dimensional quantum lens spaces for which all the weights are coprime with the order of the acting cyclic group $\Zb_r$. They managed to reduce the classification result to elementary matrix algebras using $SL_{\mathcal{P}}$-equivalence and to prove that the lowest dimension for which we get different quantum lens spaces is dimension $7$. Here they showed that there exists two different quantum lens spaces when $r$ is a multiple of 3, and precisely one when this is not the case. 

Further investigation of quantum lens spaces, as defined in \cite{hs}, was conducted in \cite{jkr} by Jensen, Klausen and Rasmussen using $SL_{\mathcal{P}}$-equivalence. For a fixed $r$ they showed how large the dimension of the quantum lens space $C(L_q(r;m_1,...m_n))$ must be to obtain non-isomorphic quantum lens spaces. The work is based on computer experiments by Eilers, who came up with a suggestion for a number $s$ such that for $n<s$ the quantum lens spaces are all isomorphic.

In this paper we will extend the result by Eilers, Restorff, Ruiz and Sørensen to quantum lens spaces of dimension less than or equal to $7$ for which $\gcd(m_i,r)\neq 1$ for one and only one $i$. The work builds on computer experiments, which were made in collaboration with Søren Eilers. We use a program written by Eilers in Maple 2019\footnote{Maplesoft, a division of Waterloo Maple Inc., Waterloo, Ontario.}, which has been optimised slightly by the present authors. Concretely, the program computes the adjacency matrices and isomorphism classes given the order $r$ and the set $\lbrace \gcd(m_i,r)\colon i=1,\dots 4 \rbrace$. By considering various combinations of the values of $r$ and the weights, we came up with a suggestion for an invariant, depending on which weight that is not coprime with the order of the acting group. In this way, experiments have played a crucial role in determining the statement of the presented theorems.

The procedure to prove this experimental observation follows by first constructing the adjacency matrices, which are presented in section \ref{adjacencymatrix}. Afterwards we calculate an invariant using $SL_{\mathcal{P}}$-equivalence which involves some long calculations. Therefore the proofs are postponed to section \ref{invariant}, and the main theorems (Theorem \ref{Thm:Main} \& \ref{5dimensional}) are stated in section \ref{classification}.
\\ 
\\
\textbf{Acknowledgements} The authors would like to thank Søren Eilers for helpful discussions as well as providing a program which has been crucial to the investigation. The authors also gratefully acknowledge helpful comments and suggestions from Wojciech Szymański and James Gabe.

\section{Preliminaries}
We recall first some concepts of graph $C^*$-algebras which are needed in this paper. For the definition of a graph $C^*$-algebra we refer to \cite{bpis,flr}. Let $E$ be a directed graph, we denote by $A_E$ the adjacency matrix for $E$ and $B_E:=A_E-I$. A graph is called \textit{finite} if it has finitely many edges and vertices. 

For a directed graph $E=(E^0,E^1,r,s)$, we recall that a vertex is \textit{regular} if $s^{-1}(v)=\{e\in E^1| \ s(e)=v\}$ is finite and nonempty, it is called \textit{singular} if this is not the case. If the graph has finitely many vertices, we say that a nonempty subset $S\subseteq E^0$ is \textit{strongly connected} if for any pair of vertices $v,w\in S$ there exists a path from $v$ to $w$. We let $\Gamma_E$ be the set of all strongly connected components and all singletons of singular vertices which are not the base of a cycle. The structure of $\Gamma_E$ will become crucial in the definition of $SL_{\mathcal{P}}$-equivalence. 

\subsection{Quantum Lens spaces}
The quantum $(2n+1)$-sphere by Vaksman and Soibelman, denoted $C(S_q^{2n+1})$, is the universal $C^*$-algebra generated by $z_0,z_1,...,z_n$ with the following relations: 
\[
\begin{aligned}
&z_iz_j=qz_jz_i, \;\;\; \text{for} \; i<j, \ \ z_iz_j^*=qz_j^*z_i, \;\;\; \text{for} \; i\neq j, \\
&z_iz_i^*=z_i^*z_i+(q^{-2}-1)\sum_{j=i+1}^n z_jz_j^*, \ \ \sum_{j=1}^n z_jz_j^*=1,
\end{aligned} 
\]
where $q\in (0,1)$, see \cite{vs}. Let $\underline{m}=(m_0,m_1,...,m_n)$ be a sequence of positive integers. The $C^*$-algebra $C(S_q^{2n+1})$ admits, by universality, an action of $\Z_r$ for any $r\in\mathbb{N}$, given by 
\[
z_i\mapsto \theta^{m_i}z_i,
\]
where $\theta$ is a generator of $\Z_r$. The quantum lens space $C(L_q^{2n+1}(r;\underline{m}))$ is defined as the fixed point algebra of $C(S_q^{2n+1})$ under this action.

It was proven in \cite{BS} that $C(L_q^{2n+1}(r;\underline{m}))$ is isomorphic to the graph $C^*$-algebra $C^*(L_{2n+1}^{r;\underline{m}})$. The graph $L_{2n+1}^{r;\underline{m}}$ is constructed using the skew product graph $L_{2n+1}\times_{\underline{m}}\Zb_r$, which has vertices $(v_i,k), i=0,...,n, k=0,...,r-1$ and edges $(e_{ij}, k), i,j=0,...,n, i\leq j, k=0,...,r-1$, with source $(v_i,k-m_i \Mod{r})$ and range $(v_j,k)$. The graph $L_{2n+1}^{r;\underline{m}}$ is constructed using the notion of admissible paths which is defined in \cite{BS} as follows:
\begin{definition}
A path from $(v_i,s)$ to $(v_j,t)$ in  $L_{2n+1}\times_{\underline{m}}\Zb_r$ is called \textit{admissible} if it does not pass through any $(v_l,k)$ for which $l=i+1,...,j-1$ and $k=0,...,\gcd(m_l,r)-1$. 
\end{definition}
\begin{remark}
Comparing with the notion of 0-simple paths from \cite[Definition 7.4]{errs}, it is clear that the 0-simple paths are exactly the admissible paths when all weights are coprime to the order of the acting group.
\end{remark}
The graph $L_{2n+1}^{r;\underline{m}}$ has vertices $v_i^b,i=0,...,n, b=0,...,\gcd(m_i,r)-1$. There are edges $e_{ij;a}^{st}$ with source $v_i^s$ and range $v_j^t$ labelled by $a=1,...,n_{ij}^{st}$, where $n_{ij}^{st}$ is the number of admissible paths from $(v_i,s)$ to $(v_j,t)$. 
\\

We will in this paper investigate quantum lens spaces $C(L_q^7(r,\underline{m}))$ for which $\gcd(m_l,r)=n$ for a single $l \in \lbrace 0,1,2,3 \rbrace$ and the remaining weights are coprime to $r$. In the process of finding an invariant for 7-dimensional quantum lens spaces we will also be able to calculate one for quantum lens spaces of dimension 5. We will in the following therefore have our focus on 7-dimensional quantum lens spaces. 

Under the above assumptions on the weights the skew product graph $L_{7}\times_{\underline{m}}\Z_r$ consists of four levels, labeled by level 0,1,2 and 3, with edges going from level $i$ to $j$ if $i<j$. At the level on which $\gcd(m_l,r)=n$ we have $n$-cycles based on each of the vertices $(v_l,k), k=0,1,...,n-1$.
The graph $L_7^{r;\underline{m}}$ consists of four levels as before and $n+3$ vertices, which are all the base of a loop. There is one vertex in each level except for level $i$ where $\gcd(m_i,r)=n\neq 1$, here we have $n$ vertices. There is at least one edge going from a lower level to a higher one, but there are no edges between vertices at the same level. We will denote the vertices by $v_i, i=1,...,n+3$ as indicated in Figure \ref{fig:Graph1}. 
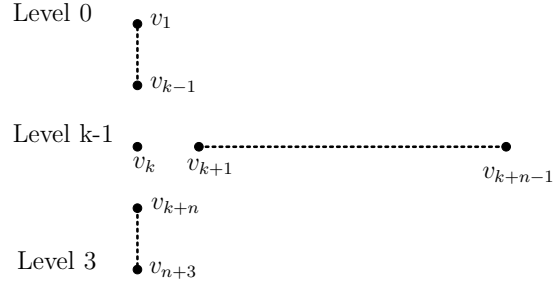
\begin{figure}[H]
\begin{center}
\resizebox{0.6\columnwidth}{!}{
\begin{tikzpicture}[line cap=round,line join=round,>=triangle 45,x=1cm,y=1cm]
\clip(-5.50780683325514,-3.5955086330893713) rectangle (6.105598048619293,3.3901303640380998);
\draw [line width=1.2pt,dotted] (-2,0)-- (3,0);
\draw (-2.9387809048404927,1.2786022036973024) node[anchor=north west] {$v_{k-1}$};
\draw (-3.2203179928859336,-0.02350682851285593) node[anchor=north west] {$v_k$};
\draw (2.469163724355169,-0.18187144053841572) node[anchor=north west] {$v_{k+n-1}$};
\draw (-2.340514592743931,-0.05869896451853588) node[anchor=north west] {$v_{k+1}$};
\draw (-2.9211848368376527,-0.6921574126207751) node[anchor=north west] {$v_{k+n}$};
\draw [line width=1.2pt,dotted] (-3,1)-- (-3,2);
\draw [line width=1.2pt,dotted] (-3,-1)-- (-3,-2);
\draw (-2.9387809048404927,2.316770215864861) node[anchor=north west] {$v_{1}$};
\draw (-2.9563769728433327,-1.7303254247883337) node[anchor=north west] {$v_{n+3}$};
\draw (0.6332208997360376,-0.12908323652989578) node[anchor=north west] {};
\draw (-5.173481541201179,2.492730895893261) node[anchor=north west] {$\text{Level 0}$};
\draw (-5.173481541201179,0.5219712795751834) node[anchor=north west] {$\text{Level k-1}$};
\draw (-5.103097269189819,-1.554364744759934) node[anchor=north west] {$\text{Level 3}$};
\begin{scriptsize}
\draw [fill=black] (-2,0) circle (2pt);
\draw [fill=black] (-3,0) circle (2pt);
\draw [fill=black] (-3,-1) circle (2pt);
\draw [fill=black] (-3,1) circle (2pt);
\draw [fill=black] (3,0) circle (2pt);
\draw [fill=black] (-3,2) circle (2pt);
\draw [fill=black] (-3,-2) circle (2pt);
\end{scriptsize}
\end{tikzpicture}}
\end{center}
    \caption{The graph $L_7^{r;\underline{m}}$}
    \label{fig:Graph1}
\end{figure}
\subsection{$SL_{\mathcal{P}}$-equivalence}
In \cite[Theorem 6.1]{errs} a classification result of graph $C^*$-algebras over finite graphs is given using the reduced filtered K-theory. It was also shown that for type I/postliminal $C^*$-algebras, we can give a classification working with $SL_{\mathcal{P}}$-equivalence instead of working directly with the reduced filtered K-theory.

We give a description of $SL_{\mathcal{P}}$-equivalence by considering the ideal structure of $C^*(L_7^{r;\underline{m}})$ when $\gcd(m_i,r)\neq 1$ for one and only one $i$. There exists by \cite[Lemma 3.16]{errs} an order preserving homeomorphism between the set of strongly connected components of $L_7^{r;\underline{m}}$, denoted $\Gamma_{L_7^{r;\underline{m}}}$, and the set of all proper ideals of $C^*(L_7^{r;\underline{m}})$ that are prime and gauge invariant, denoted Prime$_\gamma(C^*(L_7^{r;\underline{m}}))$. It can easily be seen that $\Gamma_{L_7^{r;\underline{m}}}$ is exactly equal to the collection of singleton sets of vertices. We write $\gamma_i\coloneqq \lbrace v_i \rbrace$ for each such set, and consequently we have $\Gamma_{L_7^{r;\underline{m}}}=\{\gamma_i|i=1,...,n+3\}$.

It follows immediately that $|\text{Prime}_\gamma(C^*(L_7^{r;\underline{m}}))|=n+3$. We therefore get non isomorphic quantum lens spaces for different values of $n$. The partial order on $\Gamma_{L_7^{r;\underline{m}}}$ is given as follows: $\gamma_j\leq \gamma_i$ if there is a path from $v_i$ to $v_j$. 
The set $\Gamma_{L_7^{r;\underline{m}}}$ can be illustrated by its component graphs, which are depicted in Figure \ref{Fig:Component}. In these graphs, an arrow from $\gamma_i$ to $\gamma_j$ indicates that $\gamma_i\geq \gamma_j$. 
\begin{figure}[H]
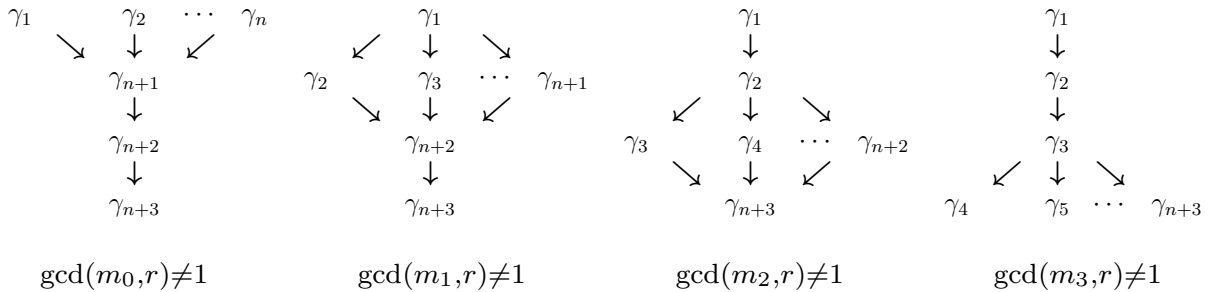

\resizebox{\columnwidth}{!}{
$
\begin{matrix}
\gamma_1 & & \gamma_2 & \cdots & \gamma_{n} \\
 & \boldsymbol{\searrow} & \boldsymbol{\downarrow}& \boldsymbol{\swarrow} &\\
& & \gamma_{n+1} & & \\
& & \boldsymbol{\downarrow} & & \\
& & \gamma_{n+2} & &\\
 & & \boldsymbol{\downarrow} & &\\
& & \gamma_{n+3} & & 
\end{matrix}
\ \ \ \
\begin{matrix}
& & \gamma_{1} & & \\
& \boldsymbol{\swarrow}  & \boldsymbol{\downarrow}  & \boldsymbol{\searrow} & \\
\gamma_{2} & & \gamma_3 & \cdots & \gamma_{n+1} \\
& \boldsymbol{\searrow} & \boldsymbol{\downarrow}  & \boldsymbol{\swarrow} & \\
& & \gamma_{n+2} & &  \\
& & \boldsymbol{\downarrow}  & & \\
& & \gamma_{n+3} & & \\
\end{matrix}
\ \ \ \
\begin{matrix}
& & \gamma_{1} & & \\
& & \boldsymbol{\downarrow}  & & \\
& & \gamma_{2} & & \\
& \boldsymbol{\swarrow}  & \boldsymbol{\downarrow}  & \boldsymbol{\searrow} & \\
\gamma_{3} & & \gamma_4 & \cdots & \gamma_{n+2} \\
& \boldsymbol{\searrow} & \boldsymbol{\downarrow}  & \boldsymbol{\swarrow} & \\
& & \gamma_{n+3} & & 
\end{matrix}
\ \ \ \
\begin{matrix}
& & \gamma_1 & & \\
& & \boldsymbol{\downarrow} & & \\
& & \gamma_2 & &\\
 & & \boldsymbol{\downarrow} & &\\
& & \gamma_3 & & \\ 
& \boldsymbol{\swarrow} & \boldsymbol{\downarrow}  & \boldsymbol{\searrow} &\\ 
\gamma_4 & & \gamma_5 & \cdots & \gamma_{n+3}
\end{matrix}
$}
\vspace{0.3cm}
\\
\resizebox{\columnwidth}{!}{$
\ \ \ {\scriptstyle \gcd(m_0,r)\neq 1} \ \ \ \ \ \ \ \ \ \ \ {\scriptstyle \gcd(m_1,r)\neq 1} \ \ \ \ \ \ \ \ \ \ \ {\scriptstyle \gcd(m_2,r)\neq 1} \ \ \ \ \ \ \ \ \ \ \ {\scriptstyle \gcd(m_3,r)\neq 1 } \ \ \
$}
\label{Fig:Component}
\caption{Component graphs of 7-dimensional quantum lens spaces.}
\end{figure}

Let $\mathcal{P}=\{1,2,...,n+3\}$ and $k$ be such that $\gcd(m_{k-1},r)=n$. We define a partial order, $\preceq$, on $\mathcal{P}$ by:
\begin{itemize}
    \item $k-1\succeq...\succeq 1$,
    \item $i\succeq k-1$ for $i=k,..,n+k-1$,
    \item $n+k\succeq i$ for $i=k,...,n+k-1$,
    \item $n+3 \succeq ... \succeq n+k$. 
\end{itemize}
The partial order satisfies that if $i\preceq j$ then $i\leq j$ which is the required assumption \cite[Assumption 4.3]{errs}.
It can easily be seen that there exists an order reversing isomorphism $\gamma_{B_{L_7^{r;\underline{m}}}}:\mathcal{P}\to \Gamma_{L_7^{r;\underline{m}}}$ mapping $i$ to $\gamma_i$. 

$SL_{\mathcal{P}}(\mathbf{1},\Z)$ is defined as the set of upper triangular matrices, $A$, with $1$ on the diagonal and which satisfies 
\[
a_{ij}\neq 0 \Rightarrow i\preceq j.
\]

$SL_{\mathcal{P}}$-equivalence simplifies in this case, since the block structure consists of $1\times 1$ matrices. Hence working with $SL_{\mathcal{P}}$-equivalence becomes a linear problem. Note that $SL_{\mathcal{P}}(\mathbf{1},\Z)$ is a group under matrix multiplication. 

\subsection{A classification result}
For a fixed $n$ we get four different classes of quantum lens spaces, one for each $i=0,1,2,3$ for which $\gcd(m_i,r)\neq 1$. By the discussion above $\text{Prime}_\gamma(C^*(L_7^{r;\underline{m}}))$ and $\text{Prime}_\gamma(C^*(L_7^{r;\underline{n}}))$ cannot be homeomorphic if $\gcd(m_i,r)\neq \gcd(n_i,r)$ for any $i$, since the structure of the ideals will be different. We will in this paper determine when two quantum lens spaces inside each of these four classes are isomorphic. 
Similarly we have three different classes of quantum lens spaces of dimension $5$ to investigate. 

To determine whether two quantum lens spaces in the same class are isomorphic we will make use of \cite[Theorem 7.1]{errs}. Since we are dealing with type I/postliminal graph $C^*$-algebras i.e. no vertex supports two distinct return paths, \cite[Lemma 7.1]{errs}, we can use the following classification result:

\begin{theorem}\cite[Theorem 7.1]{errs}
Let $E$ and $F$ be finite graphs which have no vertices supporting two distinct return paths. If $(B_E,B_F)$ is in standard form, then $C^*(E)$ and $C^*(F)$ are stably isomorphic if and only if there exist matrices $U,V\in SL_{\mathcal{P}}(\mathbf{1},\Z)$ such that $U{B_E}_{\curlywedge}V={B_F}_{\curlywedge}$. 
\end{theorem}
${B_E}_{\curlywedge}$ is obtained from the graph, where a loop has been added to all sinks in $E$. $(B_E,B_F)$ being in standard form means that the adjacency matrices have the same size and block structure, moreover they must also have the same temperatures, see \cite[Definition. 4.22]{errs} for a precise definition. When $(B_E,B_F)$  is in standard form we can work with $SL_{\mathcal{P}}$-equivalence instead of working directly with the reduced filtered K-theory, which is often more complicated to determine. For the quantum lens spaces we are investigating, the block structure consists of $1\times 1$-matrices, hence  $(B_{L_7^{r;\underline{m}}},B_{L_7^{r;\underline{n}}})$ is in standard form for two quantum lens spaces in the same class.

Since $L_{2n+1}^{r;\underline{m}}$ contains no sinks and two quantum lens spaces are stably isomorphic if and only if they are isomorphic by \cite[Proposition 14.8]{errs2}, the above theorem boils down to: 
\begin{corollary}\label{lensclassification}
If $\gcd(m_i,r)=\gcd(n_i,r)$ for each $i=0,1,2,3$ then $C^*(L_7^{r;\underline{m}})$ and $C^*(L_7^{r;\underline{n}})$ are isomorphic if and only if there exists matrices $U,V\in SL_{\mathcal{P}}(\mathbf{1},\Zb)$ such that $UB_{L_7^{r;\underline{m}}}V=B_{L_7^{r;\underline{n}}}$. 
\end{corollary}
For dimension $5$ we have a similar result by letting $\mathcal{P}=\{1,2,...,n+2\}$ and defining the order in a similar way as the one for dimension $7$. \\

Eilers, Restorff, Ruiz and Sørensen used Corollary \ref{lensclassification}, with $\mathcal{P}=\{1,2,3,4\}$ ordered linearly to completely classify the simplest case: 
\begin{theorem}\label{errsLens}\cite[Theorem 7.8]{errs}
Let $r\in \Nb$, $r\geq 2$ and let $\underline{m}=(m_0,m_1,m_2,m_3)$ and $\underline{n}=(n_0,n_1,n_2,n_3)$ be in $\Nb^4$ such that $\gcd(m_i,r)=\gcd(n_i,r)=1$ for all i. Then $C^*(L_7^{(r,\underline{m})})\cong C^*(L_7^{(r,\underline{n}}))$ if and only if 
\[(n_2^{-1}n_1-m_2^{-1}m_1)\frac{r(r-1)(r-2)}{3}\equiv 0 \Mod{r}.\]
\end{theorem}
From the above they concluded: 
\begin{corollary}\cite[Corollary 7.9]{errs}
If $3$ does not divide $r$ then 
\[
C^*(L_7^{(r,\underline{m})})\cong C^*(L_7^{(r,(1,1,1,1))})
\]
for all $\underline{m}\in \Nb^4$ with $\gcd(m_i,r)=1$. 

If $3$ divides $r$ and $\underline{m}=(m_0,m_1,m_2,m_3)\in\Nb^4$ with $\gcd(m_i,r)=1$ then 
\begin{itemize}
    \item[(i)] $C^*(L_7^{(r,\underline{m})})\cong C^*(L_7^{(r,(1,1,1,1))})$ if and only if $m_1\equiv m_2 \Mod{3}$, 
    \item[(ii)] $C^*(L_7^{(r,\underline{m})})\cong C^*(L_7^{(r,(1,1,r-1,1))})$ if and only if $m_1\not\equiv m_2 \Mod{3}$. 
\end{itemize}
\end{corollary}
For dimension less than $7$, Eilers, Restorff, Ruiz and Sørensen observed that the adjacency matrices are independent of the weights, hence all quantum lens spaces are isomorphic. We will see that this is not always the case when one of the weights is not coprime with $r$. 

\section{Classification of $C(L_{2n+1}(r;\underline{m})), n\leq 3$}\label{classification}
For a fixed value of the order of the acting group, $r$, quantum lens spaces of dimension $3$, with one and only one weight coprime with $r$, will be the same for any choice of weights. See remark \ref{dim35}. For dimension $5$ we obtain the following:
\begin{theorem}\label{5dimensional}
Let $r\in \Nb$, $r\geq 2$ and let $\underline{m}=(m_0,m_1,m_2)$ and $\underline{n}=(n_0,n_1,n_2)$ be in $\mathbb{N}^3$ such that $\gcd(r,m_k)=\gcd(r,n_k)=n$ for one $0\leq k \leq 2 $, and $\gcd(r,m_i)=\gcd(r,n_i)=1$ whenever $i\neq k$. Then 
\begin{itemize}
    \item[(i)] For $k=1$, $C^*(L_5^{(r,\underline{m})})\cong C^*(L_5^{(r,\underline{n})})$, 
    \item[(ii)] For $k=0$ or $k=2$, $C^*(L_5^{(r,\underline{m})})\cong C^*(L_5^{(r,\underline{n})})$ if and only if $m_1\equiv n_1 \Mod{n}$. 
\end{itemize}
\end{theorem}
We will now state our main theorem for quantum lens spaces of dimension $7$, which is an extension of Theorem \ref{errsLens}.
\begin{theorem}
\label{Thm:Main}
Let $r\in \Nb$, $r\geq 2$ and let  $\underline{m}=(m_0,m_1,m_2,m_3)$ and $\underline{n}=(n_0,n_1,n_2,n_3)$ be in $\mathbb{N}^4$ such that $\gcd(r,m_k)=\gcd(r,n_k)=n$ for one $0\leq k \leq 3 $, and $\gcd(r,m_i)=\gcd(r,n_i)=1$ whenever $i\neq k$. Then $C^*(L_7^{(r,\underline{m})})$ is isomorphic to $C^*(L_7^{(r,\underline{n})})$ if and only if
\begin{enumerate}
    \item $\left(n_2^{-1}n_1-m_2^{-1}m_1 \right)\frac{r(r-1)(r-2)}{3}\equiv 0\Mod{r}$ and $m_j\equiv n_j \pmod{n}, j=1,2$ if $k=0$ or $k=3$,
    \item $\left(n_2^{-1}n_1-m_2^{-1}m_1 \right)\frac{r(r-1)(r-2)}{3}\equiv 0\Mod{r}$ and $m_2\equiv n_2 \pmod{n}$ if $k=1$,
    \item $\left(n_2^{-1}n_1-m_2^{-1}m_1 \right)\frac{r(r-1)(r-2)}{3}\equiv 0\Mod{r}$ and $m_1\equiv n_1 \pmod{n}$ if $k=2$.
\end{enumerate}
\end{theorem}
The proof is postponed to section \ref{invariant}. From Theorem \ref{Thm:Main}, we may derive the following result. It is in particular interesting for computational purposes, and gives a precise determination of how many different spaces we obtain of each type.
\begin{corollary} Let $U_n$ and $U_r$ be the groups of units in $\Zb_n$ and $\Zb_r$ respectively. 
\begin{itemize}
\item[(1)] Let $\gcd(m_i,r)=n$ for $i=0$ or $i=2$ then 
\[
C^*(L_5^{(r;(m_0,m_1,m_2)}))\cong C^*(L_5^{(r;(n,k_1,1))})
\]
or
\[
C^*(L_5^{(r;(m_0,m_1,m_2)}))\cong C^*(L_5^{(r;(1,k_1,n))})
\]
respectively, where $m_1\equiv k_1 \Mod{n}$ with $k_1\in U_r$, and there are exactly $|U_n|$ isomorphism classes of quantum lens spaces. 
\item[(2)] If $3$ does not divide $r$ and $\gcd(m_i,r)=n$ for $i=0$ or $i=3$ then 
\[
C^*(L_7^{(r;(m_0,m_1,m_2,m_3)}))\cong C^*(L_7^{(r;(n,k_1,k_2,1))})
\]
and
\[
C^*(L_7^{(r;(m_0,m_1,m_2,m_3)}))\cong C^*(L_7^{(r;(1,k_1,k_2,n))})
\]
respectively, 
where $m_i\equiv k_i \Mod{n}$ with $k_i\in U_r$, and there are exactly $|U_n|^2$ isomorphism classes of quantum lens spaces. 

If $3$ divides $r$ we furthermore require that the $k_i$ satisfy 
\begin{itemize}
    \item[\tiny$\bullet$] $k_1\equiv k_2 \Mod{3}$ if $m_1\equiv m_2 \Mod{3}$
    \item[\tiny$\bullet$] $k_1\not\equiv k_2 \Mod{3}$ if $m_1\not\equiv m_2 \Mod{3}$.
\end{itemize}
In particular, there are $2|U_n|^2$ isomorphism classes if $3\not\vert n$ and  $|U_n|^2$ isomorphism classes if $3\vert n$. 
\item[(3)] If $3$ does not divide $r$ and $\gcd(m_1,r)=n$ then 
\[
C^*(L_7^{(r;(m_0,m_1,m_2,m_3)}))\cong C^*(L_7^{(r;(1,n,k_2,1))})
\]
where $m_2\equiv k_2 \Mod{n}$ with $k_2\in U_r$, and there are exactly $|U_n|$ isomorphism classes.

If $3$ divides $r$ we furthermore require: 
\begin{itemize}
    \item[\tiny$\bullet$] $k_2\equiv n \Mod{3}$ if $m_1\equiv m_2 \Mod{3}$
    \item[\tiny$\bullet$] $k_2\not\equiv n \Mod{3}$ if $m_1\not\equiv m_2 \Mod{3}$.
\end{itemize}
In particular, there are $2|U_n|$ isomorphism classes if $3\not\vert n$ and  $|U_n|$ isomorphism classes if $3\not\vert n$.

\item[(4)] If $3$ does not divide $r$ and $\gcd(m_2,r)=n$ then 
\[
C^*(L_7^{(r;(m_0,m_1,m_2,m_3)}))\cong C^*(L_7^{(r;(1,k_1,n,1))})
\]
where $m_1\equiv k_1 \Mod{n}$ with $k_1\in U_r$, and there are exactly $|U_n|$ isomorphism classes.

If $3$ divides $r$ we furthermore require
\begin{itemize}
    \item[\tiny$\bullet$] $k_1\equiv n \Mod{3}$ if $m_1\equiv m_2 \Mod{3}$
    \item[\tiny$\bullet$] $k_1\not\equiv n \Mod{3}$ if $m_1\not\equiv m_2 \Mod{3}$.
\end{itemize}
In particular, there are $2|U_n|$ isomorphism classes if $3\not\vert n$ and  $|U_n|$ isomorphism classes if $3\vert n$.
\end{itemize}
\begin{proof}
We only address the proof of (2) since the remaining follow by a similar approach.
If $3$ does not divide $r$ we notice that we by Theorem \ref{Thm:Main} only need to consider the condition $m_i\equiv n_i \Mod{n}$. It is clear that if $\gcd(n,m_i)\neq 1$, then this is also true for any integer equivalent to $m_i \Mod{n}$. It suffices to show that if $[k]_n\in U_n$, then $[k]_n$ contains an element, $k'$, such that $\gcd(k',r)=1$. Indeed, consider such a $k$. We set $p$ to be the product of $1$ and all prime factors of $r$ which are factors of neither $k$ nor $n$. Now set $k'\coloneqq np+k\equiv k \pmod{n}$ and assume that $\gcd(k',r)\neq 1$. Consider a common prime factor $q$ of $r$ and $k'$. Since $q$ divides $k'$ but only divides exactly one of $np$ and $k$ by construction, we have a contradiction, and $\gcd(k',r)=1$ as desired.

If $3$ divides $r$ then we also need to consider the first part of the invariant. We have 
\[
\frac{r(r-1)(r-2)}{3}\equiv \frac{2r}{3} \Mod{r}
\]
hence $3$ must divide $n_2^{-1}n_1-m_2^{-1}m_1$ to get isomorphic quantum lens spaces. It follows immediately that this is the case if $3$ divides $n$ since $m_i\equiv n_i \Mod{n}$ and hence $m_i\equiv n_i \Mod{3}$. If $3$ does not divide $n$ then it follows by the invariant that $m_1\equiv m_2 \Mod{3}$ if and only if $n_1\equiv n_2 \Mod{3}$. The assertion follows if for $i=1,2$, the class $[m_i]_n$ contains a $k_i$ satisfying $\gcd(k_i,r)=1$ and $m_i\not\equiv k_i \pmod{3}$. This is satisfied by choosing either $k_i\coloneqq m_i+ \frac{r}{3^{d}}$ or $k_i\coloneqq m_i+ \frac{2r}{3^{d}}$, where $d$ is the multiplicity of $3$ as a prime factor of $r$. Thus we obtain twice the number of equivalence classes in the latter case.
\end{proof}
\end{corollary}

\section{Adjacency matrices}\label{adjacencymatrix}
For each $i=0,...,n$ let $(L_{2n+1}\times_{\underline{m}}\Zb_r)\left< i\right>$ be the subgraph of $L_{2n+1}\times_{\underline{m}}\Zb_r$ with vertex set $\{v_i\}\times \Zb_r$ and edges $\{e_{ii}\}\times \Zb_r$. 
\begin{definition}\cite[Definition 7.4]{errs}
 We call an admissible path $(e_{i_1,j_1},h_1)\cdots (e_{i_l,j_l},h_l)$ in $L_{2n+1}\times_{\underline{m}} \Zb_r$ $k$-step if there exists integers $0<t_1<t_2<\cdots <t_{k+1}$ such that $t_1=i_1$ and $t_{k+1}=j_l$ and for each $2\leq \alpha\leq k$ we have 
\[
 \{r((e_{i_s,j_s},h_s))| 1\leq s\leq l\}\cap ((L_{2n+1}\times_{\underline{m}}\Zb_r)\left< t_\alpha \right>)^0\neq \emptyset, 
 \]
 and 
 \[
\{r((e_{i_s,j_s},h_s))| 1\leq s\leq l\}\subseteq \bigcup_{i=1}^{k+1} ((L_{2n+1}\times_{\underline{m}}\Zb_r)\left< t_i \right>)^0.
\]
\end{definition}
Intuitively an admissible path is $k$-step if it touches vertices from precisely $k-1$ different levels not including the level the path starts at and ends in. Below, we give formulae for the number of 1-step, 2-step and 3-step admissible paths in each relevant case. For paths that only touch levels for which the corresponding weights are coprime to the order of the acting group, we refer to \cite[Lemma 7.6]{errs}, which describes this case to completion.

We will in the following, when they exist, let $m_i^{-1}$ denote the fixed representative multiplicative inverse of $m_i$ in $\Zb_r$ for which $0\leq m_i^{-1} \leq r-1$, and $a_i$ to denote a fixed representative of the multiplicative inverse of $m_i$ in $\Zb_n$ for which $0\leq a_i\leq n-1$.

\begin{lemma}{\textbf{1-step admissible paths}}
Let $r\in \Nb$, $r\geq 2$ and let $\underline{m}=(m_0,m_1,m_2,m_3)\in \Nb^4$ be such that for a single $\ell\in \lbrace 0,1,2,3\rbrace$, $\gcd(m_\ell,r)=n$ and $\gcd(m_i,r)=1$ for $i\neq \ell$. Let $t\in \lbrace 0,\dots, n-1 \rbrace$ then 
\begin{enumerate}
    \item For $i<\ell$, there are $\frac{r}{n}$ 1-step admissible paths from $(v_i,0)$ to $(v_\ell,t)$.
    \item For $i>\ell$, there are $\frac{r}{n}$ 1-step admissible paths from $(v_\ell,t)$ to $(v_i,0)$.
\end{enumerate}
\begin{proof}
This follows since at the $\ell$'th level we have $n$ loops each going through only one of the $(v_l,i)$, $i=0,1,...,n-1$.
\end{proof}
\end{lemma}

\begin{lemma}{\textbf{2-step admissible paths}}\label{2step}
Let $r\in \Nb$, $r\geq 2$ and let $\underline{m}=(m_0,m_1,m_2,m_3)\in \Nb^4$ be such that for a single $\ell\in \lbrace 0,1,2,3\rbrace$, $\gcd(m_\ell,r)=n$ and $\gcd(m_i,r)=1$ for $i\neq \ell$. Let $t\in \lbrace 0,\dots, n-1 \rbrace$.
\begin{enumerate}
    \item If $i<k<\ell-1$, there are $\frac{r(r-n)}{2n}+\left(a_k\cdot t-1+q_t n\right)\frac{r}{n}$ 2-step admissible paths from $(v_i,0)$ to $(v_\ell,t)$ passing through the $k$'th level for some $q_t\in \Zb$ if $t\neq 0$ and $\frac{r(r+n-2)}{2n}$ if $t=0$;
    
    \item If $i>k>\ell-1$, there are $\frac{r(r-n)}{2n}-\left(a_k\cdot t-1+q_tn\right)\frac{r}{n}$ 2-step admissible paths from $(v_\ell,t)$ to $(v_i,0)$ passing through the $k$'th level for some $q_t\in \Zb$ if $t\neq 0$ and $\frac{r(r-n)}{2n}$ if $t=0$;
    \item $If$ $i<\ell<j$, there are $\frac{r(r-n)}{2n}$ paths from $(v_i,0)$ to $(v_j,0)$ passing through the $\ell$'th level.
\end{enumerate}
\begin{proof}
$(1)$ First note that there is only one path from $(v_i,0)$ to $(v_k,sm_k)$ for $s=1,2,...,r-1$ not coming back to $(v_i,0)$ and not going through any vertices at the k-th level. We have an edge from $(v_k,sm_k)$ into the cycle containing $(v_\ell,t)$ if and only if 
\[
m_k(s+1)\equiv t \Mod{n}.
\]
Equivalently, $s\equiv a_kt-1 \Mod{n}$. Let $q_t\in \Z$ be such that $s_t:=a_kt-1+q_tn$ is an integer between $0$ and $n-1$. Hence $(v_k,s_tm_k)$ is the first vertex in level k which has a path ending in the cycle containing $(v_{\ell},t)$. The number of paths from $(v_k,sm_k)$ for $s=1,2,...,r-1$ to $(v_{\ell},t)$ is then 
\[
\begin{aligned}
&\frac{r-(s-h)}{n}, \ \text{if} \ s\equiv h=0,...,s_t \Mod{n}, \\
& \frac{r-(s-h)}{n}-1, \ \text{if} \ s\equiv h= s_t+1,...,n-1\Mod{n}. 
\end{aligned}
\]
The number of 2-step admissible paths then becomes 
\[
\begin{aligned}
&\sum_{h=0}^{s_t}\sum_{\substack{s=1, \\ s\equiv h \Mod{n}}}^{r-1} \frac{r-(s-h)}{n} + \sum_{h=s_t+1}^{n-1} \sum_{\substack{s=1, \\ s\equiv h \Mod{n} }}^{r-1} \left(\frac{r-(s-h)}{n}-1\right) \\
&
=\sum_{h=0}^{n-1}\sum_{\substack{s=1, \\ s\equiv h \Mod{n}}}^{r-1} \frac{r-(s-h)}{n}- \sum_{h=s_t+1}^{n-1} \sum_{\substack{s=1, \\ s\equiv h \Mod{n} }}^{r-1}1 \\
&= \sum_{j=1}^{\frac{r-n}{n}} n\frac{r-jn}{n} + (n-1)\frac{r}{n} - (n-1-s_t)\frac{r}{n} \\
&=  \frac{r(r-n)}{2n} + (a_kt-1+q_tn)\frac{r}{n}.
\end{aligned}
\]
If $t=0$ then $q_0=1$ hence the number of 2-step admissible paths becomes
\[
\frac{r(r-n)}{2n}+(n-1)\frac{r}{n}=\frac{r^2-rn+2nr-2r}{2n}=\frac{r(r+n-2)}{2n}.
\]
(2) Let $t>0$. First, we have precisely one edge from $(v_{\ell},t)$ to $(v_k,m_kh)$ if and only if $m_kh\equiv t \Mod{n}$ for $h=1,2,\dots,r-1$. The number of paths from $(v_k,m_kh)$ to $(v_i,0)$ is $r-h$. Hence the total number of admissible 2-step paths is 
\[
\begin{aligned}
\sum_{\substack{h=1, \\ h\equiv a_kt \Mod{n}}}^{r-1} (r-h)
&= \sum_{j=0}^{\frac{r-n}{n}} (r-(a_kt+q_tn+jn)) = \frac{r(r+n)}{2n}-(a_kt+q_tn)\frac{r}{n} \\
&= \frac{r(r-n)}{2n}+\frac{r}{n}-(a_kt+q_tn)\frac{r}{n}
=\frac{r(r-n)}{2n}-(a_kt-1+q_tn)\frac{r}{n},
\end{aligned}
\]
where $q_t\in\mathbb{Z}$ is such that $0<a_kt+q_tn<n$. If $t=0$ then the number of 2-step admissible paths becomes 
\[
\sum_{j=1}^{\frac{r-n}{n}}(r-jn)=\frac{r(r-n)}{2n}. 
\]
(3)  Note that for each $t\in \lbrace 1,\dots,n-1\rbrace$ and for each $h\in\lbrace  1, \dots, \frac{r}{n}-1 \rbrace$, there is precisely one edge from $(v_i,0)$ to $(v_\ell,t+m_\ell h)$, and the number of admissible paths from $(v_\ell,t+m_\ell h)$ to $(v_j,0)$ is $\frac{r}{n}-h$. Thus the number of admissible 2-step paths is

\[\sum_{t=0}^{n-1}\sum_{h=1}^{\frac{r}{n}-1} \frac{r}{n}-h = \frac{r(r-n)}{2n}. \]
\end{proof}
\end{lemma}    
\begin{remark}\label{dim35}
For quantum lens spaces of dimension 3, we see immediately by Lemma \ref{2step} that the adjacency matrices for a fixed $r$ and $n$, will all be the same. For quantum lens spaces of dimension $5$ the adjacency matrices are given by the following: 
\begin{center}
\begin{tabular}{c c c c}
$\gcd(m_0,r)=n$ & $\gcd(m_1,r)=n$ & &  $\gcd(m_2,r)=n$\\ 
$\begin{psmallmatrix}
 &   &  &\frac{r}{n} & y_0 
\\
 &  I_n &  & \vdots & \vdots 
\\
 &  &  & \frac{r}{n} & y_{n-1} 
\\   0 & \dots & 0 & 1 & r 
\\   0 & \dots & 0 & 0 & 1 
\end{psmallmatrix}$  &
$\begin{psmallmatrix}
1 & \frac{r}{n} & \frac{r}{n} & \dots & \frac{r}{n} & \frac{r(r+n)}{2n} \\
0 &  &  &  &  & \frac{r}{n}  \\
0 &  &  &  &  & \frac{r}{n} 
\\
0 & & & I_n  & & \vdots 
\\
\vdots & & & & & \frac{r}{n} 
\\
0 & 0 & 0  & \dots &  0 &  1 
\end{psmallmatrix}$ & &
$\begin{psmallmatrix}
1 & r & z_0 & \cdots & \cdots & z_{n-1}\\
0 & 1 & r/n & \cdots & \cdots & r/n  \\
0 & 0 &  &  &  & \\
\vdots & \vdots &  & I_{n} &  &  \\
0 & 0 &  &  &  &  \\
 0 & 0 &  &  &  & 
\end{psmallmatrix}$

\\
$y_0=\frac{r}{n}+\frac{r(r-n)}{2n}$, & & & 
$z_0=\frac{r(r+n)}{n}$, \\
$y_t=\frac{r(r-n)}{2n}-\frac{r}{n}(a_1-t-2)-q_tr$ & & & $z_t=\frac{r(r-n)}{2n}+a_1t+q_tr$
\end{tabular}
\end{center}
\end{remark}

We will now calculate the adjacency matrices for $7$-dimensional quantum lens spaces. 
\begin{lemma} Let $r\in \Nb$, $r\geq 2$ and let $\underline{m}=(m_0,m_1,m_2,m_3)$ be such that $\gcd(r,m_3)=n$ and $\gcd(r,m_i)=1,i\neq 3$. Then we may for each $0\leq l < r-1$ and each $0\leq t \leq n-1$ find $k_l,s_t\in \Zb$ such that
\[
A_{L_q^{7}(r;\underline{m})}= 
\begin{psmallmatrix}
1 & r & \frac{r(r+1)}{2} & x_0 & \cdots & \cdots & x_{n-1} \\
0 & 1 & r & y_0 & \cdots & \cdots & y_{n-1}\\
0 & 0 & 1 & r/n & \cdots & \cdots & r/n  \\
0 & 0 & 0 &  &  &  & \\
\vdots & \vdots & \vdots &  & I_{n} &  &  \\
0 & 0 & 0 &  &  &  &  \\
0 & 0 & 0 &  &  &  & 
\end{psmallmatrix}
\]
where
\[y_0=\frac{r(r+n)}{2n}\]
and
\[
x_0\equiv -m_2^{-1}m_1\frac{r(r-2)(r-1)}{3n}+\sum_{l=1}^{r-2} l\frac{r(1-k_l)}{n} +\sum_{h=0}^{n-1}\sum_{\mathclap{{\substack{l=1 \\l\equiv m_2a_1h-1\Mod{n}}}}}^{r-2}\frac{lh}{n} - \frac{r}{n}\Mod{r}. 
\]
For $t\geq 1$ we have 
\[y_t=\frac{r(r-n)}{2n}+a_2t\frac{r}{n} + rq_t\]
and
\[\begin{aligned}x_t&\equiv-m_2^{-1}m_1\frac{r(r-2)(r-1)}{3n} + \sum_{l=1}^{r-2} l\frac{r(1-k_l)}{n} \\ &+\mathop{\sum_{h=0}^{n-1}\sum_{l=1}^{r-2}}_{l\equiv m_2a_1h-1\Mod{n}}\frac{lh}{n}  - \mathop{\sum_{h=s_t+1}^{n-1}\sum_{l=1}^{r-2}}_{l\equiv m_2a_1h-1\Mod{n}}l +\frac{r}{n}\left(a_2t+a_1t-1 \right) \Mod{r}
\end{aligned}
\]
\begin{proof}
We will now calculate the number of 3-step admissible paths from $(v_0,0)$ to $(v_3,t)$ for $t=0,1,2,3$. First notice that there are exactly $l$ paths from $(v_0,0)$ to $(v_1,lm_1)$ not coming back to $(v_0,0)$, and exactly one edge from $(v_1,lm_1)$ to $(v_2,(l+1)m_1)$, hence we wish to find the number of paths from $(v_2,(l+1)m_1)$ to $(v_3,t)$, denoted $P_l$. Then the total number of 3-step admissible paths will be $\sum_{l=1}^{r-2}lP_l$.  

We can express $(v_2,(l+1)m_1)$ as $(v_2,sm_2)$, where $0<s\leq r$ satisfies $m_2s\equiv (l+1)m_1 \Mod{r}$.  As in the proof of Lemma \ref{2step}(1), we let $s_t$ denote a representative of the class $[a_2t-1]_n$ which lies between $0$ and $n-1$, and let $k_l$ be an integer such that $0 <m_2^{-1}m_1(l+1)+rk_l<r $. By reasoning as in Lemma \ref{2step} (1) we have 
\[
P_l=\frac{r-(m_2^{-1}m_1(l+1)+rk_l-h)}{n}
\]
if $s\equiv h=0,...,s_t \Mod{n}$ i.e. $l\equiv m_2a_1h-1 \Mod{n}$ and 
\[
P_l=\frac{r-(m_2^{-1}m_1(l+1)+rk_l-h)}{n}-1
\]
if $s\equiv h=s_t+1,...,n-1 \Mod{n}$. The number of 3-step admissible paths becomes 
\vspace{0.2cm} 
\\
\resizebox{\columnwidth}{!}{$
\begin{aligned}
\sum_{l=1}^{r-2}lP_l&=
\sum_{h=0}^{s_t}\sum_{\mathclap{{\substack{l=1 \\l\equiv m_2a_1h-1\Mod{n}}}}}^{r-2}l\frac{r-(m_2^{-1}m_1(l+1)+k_lr-h)}{n} +
\sum_{h=s_t+1}^{n-1}\sum_{\mathclap{{\substack{l=1 \\l\equiv m_2a_1h-1\Mod{n}}}}}^{r-2}l\left(\frac{r-(m_2^{-1}m_1(l+1)+k_lr-h)}{n}-1\right) \\
&=\sum_{h=0}^{n-1}\sum_{\mathclap{{\substack{l=1 \\l\equiv m_2a_1h-1\Mod{n}}}}}^{r-2}l\frac{r-(m_2^{-1}m_1(l+1)+k_lr-h)}{n} - \sum_{h=s_t+1}^{n-1}\sum_{\mathclap{{\substack{l=1 \\l\equiv m_2a_1h-1\Mod{n}}}}}^{r-2}l
\\
&=-\sum_{h=0}^{n-1}\sum_{\mathclap{{\substack{l=1 \\l\equiv m_2a_1h-1\Mod{n}}}}}^{r-2}\frac{m_2^{-1}m_1(l+1)l}{n} + \sum_{h=0}^{n-1}\sum_{\mathclap{{\substack{l=1 \\l\equiv m_2a_1h-1\Mod{n}}}}}^{r-2}l\frac{r(1-k_l)}{n}  +\mathop{\sum_{h=0}^{n-1}\sum_{l=1}^{r-2}}_{l\equiv m_2a_1h-1\Mod{n}}\frac{lh}{n}  - \mathop{\sum_{h=s_t+1}^{n-1}\sum_{l=1}^{r-2}}_{l\equiv m_2a_1h-1\Mod{n}}l
\\
&=-\sum_{l=1}^{r-2}\frac{m_2^{-1}m_1(l+1)l}{n} + \sum_{l=1}^{r-2}l\frac{r(1-k_l)}{n}  +\mathop{\sum_{h=0}^{n-1}\sum_{l=1}^{r-2}}_{l\equiv m_2a_1h-1\Mod{n}}\frac{lh}{n}  - \mathop{\sum_{h=s_t+1}^{n-1}\sum_{l=1}^{r-2}}_{l\equiv m_2a_1h-1\Mod{n}}l
\\
&= -m_2^{-1}m_1\frac{r(r-2)(r-1)}{3n}+\sum_{l=1}^{r-2} l\frac{r(1-k_l)}{n} +\mathop{\sum_{h=0}^{n-1}\sum_{l=1}^{r-2}}_{l\equiv m_2a_1h-1\Mod{n}}\frac{lh}{n}  - \mathop{\sum_{h=s_t+1}^{n-1}\sum_{l=1}^{r-2}}_{l\equiv m_2a_1h-1\Mod{n}}l. 
\end{aligned}
$}
\vspace{0.2cm}
\\
Adding up the 1-step, 2-step and 3-step admissible paths we arrive at the above adjacency matrix, here we also make use of \cite[Lemma 7.6 (i), (ii)]{errs}. 
\end{proof}
\end{lemma}

\begin{lemma}
\label{Lemma:m0}
Let $r\in \Nb$, $r\geq 2$ and let $\underline{m}=(m_0,m_1,m_2,m_3)$ be such that $\gcd(m_0,r)=n$ and $\gcd(m_i,r)=1, i\neq 0$. Then we may for each $0\leq l < r-1$ and each $0\leq t \leq n-1$ find $k_l,c_t,q_t\in \Zb$ such that
\[
A_{L_q^{7}(r;\underline{m})}=
\begin{psmallmatrix}
 &   &  &\frac{r}{n} & y_0 & x_0
\\
 &  I_n &  & \vdots & \vdots & \vdots 
\\
 &  &  & \frac{r}{n} & y_{n-1} & x_{n-1}
\\   0 & \dots & 0 & 1 & r & \frac{r(r+1)}{2}
\\   0 & \dots & 0 & 0 & 1 & r
\\   0 & \dots & 0 & 0 & 0 & 1
\end{psmallmatrix}
\]
where  
\[y_0=\frac{r(r-n)}{2n}+\frac{r}{n}\]
and 
\[
\begin{aligned}
x_0&\equiv -m_2^{-1}m_1\frac{r(r-2)(r-1)}{3n}+ \sum_{l=0}^{r-2} l\frac{r(1-k_l)}{n} \\ &-\sum_{h=0}^{n-1} \sum_{\mathclap{{\substack{l=1 \\l\equiv a_1h\pmod{n}}}}}^{r-2} \frac{a_1h+nq_h}{n}\left(r-m_2^{-1}m_1(l+1)-rk_l \right) \Mod{r}.
\end{aligned}
\]
For $t\geq 1$ we have 
\[
y_t=\frac{r(r-n)}{2n}-\frac{r}{n}(a_1t-2)-rq_t
\]
and
\vspace{0.2cm}
\\
\resizebox{\columnwidth}{!}{
$\begin{aligned} 
x_t &\equiv -m_2^{-1}m_1\frac{r(r-2)(r-1)}{3n}+ \sum_{l=0}^{r-2} l\frac{r(1-k_l)}{n} -\sum_{h=0}^{n-1} \sum_{\mathclap{{\substack{l=1 \\l\equiv a_1h\pmod{n}}}}}^{r-2} \frac{a_1h+nq_h}{n}\left(r-m_2^{-1}m_1(l+1)-rk_l \right) \\&+\sum_{h=c_t}^{n-1} \sum_{\mathclap{{\substack{l=1 \\l\equiv a_1h\pmod{n}}}}}^{r-2} \left(r-m_2^{-1}m_1(l+1)-rk_l \right) - \frac{r}{n}(a_2t+a_1t-3) \pmod{r}.
\end{aligned}
$}
\end{lemma}
\begin{proof} 
We will now calculate the number of 3-step admissible paths from $(v_0,t)$ to $(v_3,0)$. There is an edge from from $(v_0,t)$ to $(v_1,m_1l)$ only if $lm_1\equiv t \Mod{n}$.
Let $c_t$ be such that $(v_1,c_tm_1)$ is the first vertex in the 1st level which is connected to the cycle coming from $(v_0,t)$. For $0\leq h< n$ let $q_h$ be an integer such that $0<a_1h+nq_h<n$.
Then $c_t=a_1t+nq_t$ and the number of paths from $(v_0,t)$ to $(v_1,lm_1)$ is given by
\[
\frac{l-(a_1h+nq_h)}{n}
\]
if $lm_1\equiv h \Mod{n}$ for $0\leq h<c_t$ and 
\[
\frac{l-(a_1h+nq_h)}{n}+1
\]
if $lm_1\equiv h \Mod{n}$ for $c_t\leq h<n$. 
There is precisely one edge from $(v_1,lm_1)$ to $(v_2,m_1(l+1))$. We can express $(v_2,l(m_1+1))$ as $(v_2,sm_2)$ i.e. $m_1(l+1)\equiv sm_2 \Mod{r}$. Let $k_l$ be such that $0<m_2^{-1}m_1(l+1)+rk_l<n$. Then the number of paths from $(v_1,lm_1)$ to $(v_3,0)$ is $r-(m_2^{-1}m_1(l+1)+rk_l).$ The total number of 3-step admissible paths becomes: 
\[
\begin{aligned}
&\sum_{h=0}^{c_t-1} \sum_{\mathclap{{\substack{l=1 \\l\equiv a_1h\pmod{n}}}}}^{r-2} \frac{l-(a_1h+nq_h)}{n}\left(r-m_2^{-1}m_1(l+1)-rk_l \right)\\ 
&+\sum_{h=c_t}^{n-1} \sum_{\mathclap{{\substack{l=1 \\l\equiv a_1h\pmod{n}}}}}^{r-2} \left(\frac{l-(a_1h+nq_h)}{n}+1\right)\left(r-m_2^{-1}m_1(l+1)-rk_l \right) \\
&= -m_2^{-1}m_1\frac{r(r-2)(r-1)}{3n}+ \sum_{l=0}^{r-2} l\frac{r(1-k_l)}{n} \\ &-\sum_{h=0}^{n-1} \sum_{\mathclap{{\substack{l=1 \\l\equiv a_1h\pmod{n}}}}}^{r-2} \frac{a_1h+nq_h}{n}\left(r-m_2^{-1}m_1(l+1)-rk_l \right)+\sum_{h=c_t}^{n-1} \sum_{\mathclap{{\substack{l=1 \\l\equiv a_1h\pmod{n}}}}}^{r-2} \left(r-m_2^{-1}m_1(l+1)-rk_l \right) \\
\end{aligned}
\]
For $t=0$ we have $c_t=n$ hence the number of 3-step admissible paths is
\[
-m_2^{-1}m_1\frac{r(r-2)(r-1)}{3n}+ \sum_{l=0}^{r-2} l\frac{r(1-k_l)}{n} -\sum_{h=0}^{n-1} \sum_{\mathclap{{\substack{l=1 \\l\equiv a_1h\pmod{n}}}}}^{r-2} \frac{a_1h+nq_h}{n}\left(r-m_2^{-1}m_1(l+1)-rk_l \right).
\]
\end{proof}

\begin{lemma}\label{lemmam2}
Let $r\in \Nb$, $r\geq 2$ and let $\underline{m}=(m_0,m_1,m_2,m_3)$ be such that $\gcd(m_2,r)=n$ and $\gcd(m_i,r)=1, i\neq 2$. Then 
\[
A_{L_q^{7}(r;\underline{m})}=
\begin{psmallmatrix}
1 & r & \frac{r(r+n)}{2n} & \frac{r(r-n)}{2n}+ a_1\frac{r}{n} & \dots & \frac{r(r-n)}{2n}+ a_1(n-1)\frac{r}{n}  & x \\
0 & 1 & \frac{r}{n} & \frac{r}{n} & \dots & \frac{r}{n} & \frac{r(r+n)}{2n}\\
0 & 0 & & & & & \frac{r}{n}
\\
\vdots & \vdots & & & I_{n} & & \vdots
\\
0 & 0 &  & & & &\frac{r}{n}
\\
0 & 0 & 0& 0& \cdots & 0 & 1
\end{psmallmatrix}
\]
where
\[
x\equiv -m_1^{-1}m_2\frac{r(2r-n)(r-n)}{6n^2}+m_1^{-1}\frac{r(r-n)(n-1)}{4n}+\frac{r(r-1)}{2} \Mod{r}.
\]
\begin{proof}
We will now calculate the number of 3-step admissible paths from $(v_0,0)$ to $(v_0,3)$. First we have $m_1^{-1}m_2l-1+tm_1^{-1}+rs_t$ paths from $(v_0,0)$ to each $(v_1,lm_2-m_1+t)$ for $t=0,..,n-1$, where $s_l$ is such that $0<m_2^{-1}m_1l-1+tm_1^{-1}+rs_l<r$. $(v_1,lm_2-m_1+t)$ is connected to $(v_2,lm_2+t)$ by a single edge and there are $\frac{r}{n}-l$ paths from $(v_2,lm_2+t)$ to $(v_3,0)$. The total number of 3-step admissible paths becomes 
\[
\begin{aligned}
&\sum_{t=0}^{n-1}\sum_{l=1}^{\frac{r-n}{n}} \left(m_1^{-1}m_2l-1+tm_1^{-1}+rs_l \right) \left( \frac{r}{n}-l \right)
\\
&\equiv n \sum_{l=1}^{\frac{r-n}{n}} -l\left(m_1^{-1}m_2l-1 \right)+\sum_{t=0}^{n-1}\sum_{l=1}^{\frac{r-n}{n}} tm_1^{-1}\left( \frac{r}{n}-l \right) \Mod{r}
\\
&\equiv -m_1^{-1}m_2 \frac{r(2r-n)(r-n)}{6n^2} + \frac{r(r-n)}{2n}+\sum_{t=0}^{n-1} tm_1^{-1}\frac{r(r-n)}{2n^2} \Mod{r}
\\
&\equiv -m_1^{-1}m_2 \frac{r(2r-n)(r-n)}{6n^2} + \frac{r(r-n)}{2n}+ m_1^{-1}\frac{r(r-n)(n-1)}{4n} \Mod{r}.
\end{aligned} 
\]
\end{proof}

\end{lemma}

We state the final case without proof, as the proof is similar to that of Lemma \ref{lemmam2}. 
\begin{lemma}\label{lemmam1}
Let $r\in \Nb$, $r\geq 2$ and let $\underline{m}=(m_0,m_1,m_2,m_3)$ be such that $\gcd(m_1,r)=n$ and $\gcd(m_i,r)=1, i\neq 1$. Then 
\[
A_{L_q^{7}(r;\underline{m})}=
\begin{psmallmatrix}
1 & \frac{r}{n} & \frac{r}{n} & \dots & \frac{r}{n} & \frac{r(r+n)}{2n} & x \\
0 &  &  &  &  & \frac{r}{n} & \frac{r(r-n)}{2n}+\frac{r}{n} \\
0 &  &  &  &  & \frac{r}{n} & \frac{r(r-n)}{2n}-\frac{r}{n}a_2+\frac{r}{n}
\\
0 & & & I_n  & & \vdots & \vdots
\\
\vdots & & & & & \frac{r}{n} & \frac{r(r-n)}{2n}-\frac{r}{n}a_2(n-1)+\frac{r}{n}
\\
0 & 0 & 0  & \dots &  0 &  1 & r
\\
0 & 0 & 0& 0& \cdots & 0 & 1
\end{psmallmatrix}
\]
where
\[
x\equiv -m_2^{-1}m_1\frac{r(2r-n)(r-n)}{6n^2}+m_2^{-1}\frac{r(r-n)(n-1)}{4n}+\frac{r(r-1)}{2} \Mod{r}.
\]
\end{lemma}

\section{The invariant}\label{invariant}
We are now ready to proof Theorem \ref{Thm:Main}. The proof of Theorem \ref{5dimensional} follows by a similar but much easier approach. 
\begin{proof}[Proof of Theorem \ref{Thm:Main} (1)] We will prove the result in the case where $k=3$, the proof for $k=0$ follows by a similar approach. 

Assume that the $C^*$-algebras coming from the weights $\underline{m}=(m_0,m_1,m_2,m_3)$ and $\underline{n}=(n_0,n_1,n_2,n_3)$ are isomorphic. We denote by $x_0,...,x_{n-1}, y_0,...,y_{n-1}$ and $x_0',...,x_{n-1}'$, $y_0',...,y_{n-1}'$ the elements $x_i$ and $y_i$ coming from Lemma \ref{Lemma:m0} corresponding to $\underline{m}$ and $\underline{n}$ respectively. Then by \cite[Theorem 7.1]{errs} there exists $U,V\in SL_{\mathcal{P}}(\textbf{1},\mathbb{Z})$ such that $UB_{\underline{m}}V=B_{\underline{n}}$ where $\mathcal{P}=\{1,2,....,,n,n+1,n+2,n+3\}$ with the order $3>2>1$, $4+i>3$ for $i=0,1,...,n-1.$ We then get the following equations 
\begin{equation}\label{matrixequa1}
    y_j'\equiv y_j+u_{23}\frac{r}{n}+rv_{3,j+4}, \ \ \ \
x_j'\equiv x_j+u_{12}y_j+v_{3,j+4}\frac{r(r+1)}{2}+u_{13}\frac{r}{n} \Mod{r},
\end{equation}
for $j=0,1,...,n-1$ where $u_{lm}, v_{lm}\in\mathbb{Z}$ are the entries of $U$ and $V$ respectively. 

Since $y_0=y_0'$, $n$ divides $u_{23}$. Let $a_2=m_2^{-1} \Mod{n}$, $a_2'=n_2^{-1} \Mod{n}$, $a_1=m_1^{-1} \Mod{n}$ and $a_1'=n_1^{-1} \Mod{n}$. Then by (\ref{matrixequa1})
\[
a_2'\frac{r}{n}\equiv a_2\frac{r}{n}+u_{23}\frac{r}{n} \Mod{r},
\]
hence there exists a $k\in\Z$ such that 
\[
(a_2'-a_2)\frac{r}{n}=u_{23}\frac{r}{n}+kr.
\]
Then $\frac{a_2'-a_2}{n}\in \Z$ and $a_2'=a_2$ which implies that $m_2\equiv n_2 \Mod{n}$. 

We now consider the sum of all the $x_i's$. We have 
\[
\begin{aligned}
x_0+x_1+...+x_{n-1}&\equiv -m_2^{-1}m_1\frac{r(r-2)(r-1)}{3}+ n\sum_{h=0}^{n-1}\sum_{\mathclap{{\substack{l=1 \\l\equiv m_2a_1h-1\pmod{n}}}}}^{r-2}\frac{lh}{n} \\ &- \sum_{t=1}^{n-1} \sum_{h=s_t+1}^{n-1}\sum_{\mathclap{{\substack{l=1 \\l\equiv m_2a_1t-1\pmod{n}}}}}^{r-2}l +\sum_{t=1}^{n-1}\frac{r}{n}\left(a_1t+a_2t) \right)\pmod{r} \\
&\equiv 
-m_2^{-1}m_1\frac{r(r-2)(r-1)}{3}+\sum_{h=0}^{n-1}\sum_{\mathclap{{\substack{l=1 \\l\equiv m_2a_1h-1\pmod{n}}}}}^{r-2}lh \\ &- \sum_{t=1}^{n-1} \sum_{h=s_t+1}^{n-1}\sum_{\mathclap{{\substack{l=1 \\l\equiv m_2a_1t-1\pmod{n}}}}}^{r-2}l 
+ \frac{r(n-1)}{2}(a_1+a_2) \pmod{r}.
\end{aligned}
\]
The last term is always congruent to $0$ modulo $r$, indeed if $n$ is odd we are done, if $n$ is even then $a_1+a_2$ is even. 

Since each $s_t$ corresponds uniquely to a number between $0$ and $n-1$, we may reiterate the penultimate sum accordingly:
\[
\begin{aligned}
\mathop{\sum_{t=1}^{n-1} \sum_{h=s_t+1}^{n-1}\sum_{l=1}^{r-2}}_{l\equiv m_2a_1h-1\Mod{n}}l = \mathop{\sum_{t=1}^{n-1}
\sum_{h=t}^{n-1}\sum_{l=1}^{r-2}}_{l\equiv m_2a_1h-1\Mod{n}}l  
=\sum_{h=1}^{n-1}\sum_{\mathclap{{\substack{l=1 \\l\equiv m_2a_1h-1\Mod{n}}}}}^{r-2}hl. 
\end{aligned}
\]
Hence 
\[
x_0+x_1+...+x_{n-1}\equiv -m_2^{-1}m_1\frac{r(r-2)(r-1)}{3} \Mod{r}.
\]
Using (\ref{matrixequa1}) and the fact that $(a_2-1)\frac{r(n-1)}{2}\equiv 0 \Mod{r}$, we have 
\vspace{0.2cm}
\\
\resizebox{\columnwidth}{!}{$
\begin{aligned}
&(x_0'+x_1'+...+x_{n-1}')-(x_0+x_1+...+x_{n-1})\equiv u_{12}(y_0+y_1+...+y_{n-1})\\ &+nu_{13}\frac{r}{n} + (v_{34}+v_{35}+...+v_{3,n+3})\frac{r(r+1)}{2} \Mod{r} \\
&\equiv u_{12}\left(\frac{r(r+n)}{2n}+ (n-1)\frac{r(r-n)}{2n}+a_2\frac{r}{n}\sum_{t=1}^{n-1} t\right)+(v_{34}+v_{35}+...+v_{3,n+3})\frac{r(r+1)}{2} \Mod{r}  \\
&\equiv
u_{12}\frac{r(r+1)}{2}+(v_{34}+v_{35}+...+v_{3,n+3})\frac{r(r+1)}{2} \Mod{r} \\
&\equiv
(v_{34}+v_{35}+...+v_{3,n+3})\frac{r(r+1)}{2}\Mod{r} \equiv 0 \Mod{r}.
\end{aligned}
$} 
\vspace{0.2cm}
\\
The last congruence follows by \cite[the proof of Theorem 7.8]{errs}. Hence
\[
\left(m_2^{-1}m_1-n_2^{-1}n_1\right)\frac{r(r-1)(r-2)}{3}\equiv 0 \Mod{r}. \]
We now wish to compute the last part of the invariant which is $m_1\equiv n_1 \Mod{n}$ i.e. $a_1=a_1'$. 
For $h=s_t+1,...,n-1$ let $p_h\in\Z$ be such that $0<m_2a_1h-1+np_h<n$. First we need to expand the following sum: 
\vspace{0.2cm}
\\
\resizebox{\columnwidth}{!}{
$
\begin{aligned}
&\mathop{\sum_{h=s_t+1}^{n-1}\sum_{l=1}^{r-2}}_{l\equiv m_2a_1h-1\Mod{n}}l 
= \sum_{h=s_t+1}^{n-1}\sum_{k=0}^{\frac{r-n}{n}}\left( m_2a_1h-1+np_h+nk\right) = \sum_{h=s_t+1}^{n-1}\left(\frac{r}{n}(m_2a_1h-1)+rp_h+\frac{r(r-n)}{2n} \right) \\
&\equiv \frac{r}{n}m_2a_1\left(\frac{n(n-1)}{2}+\frac{a_2t(1-a_2t)}{2}\right)+\frac{r}{n}(a_2t-n)+(n-a_2t)\frac{r(r-n)}{2n} \Mod{r}.
\end{aligned}
$}
\vspace{0.2cm}
\\
We will now find an expression for $(x_0'-x_t')-(x_0-x_t)$ and then consider the expressions for $t=1$ and $t=n-1$. From the above we have
\vspace{0.2cm}
\\
\resizebox{\columnwidth}{!}{
$
\begin{aligned}
(x_0'-x_t')-(x_0-x_t)&\equiv \mathop{\sum_{h=s_t+1}^{n-1}\sum_{l=1}^{r-2}}_{l\equiv m_2a_1'h-1\Mod{n}}l -\mathop{\sum_{h=s_t+1}^{n-1}\sum_{l=1}^{r-2}}_{l\equiv m_2a_1h-1\Mod{n}}l + \frac{r}{n}(a_1-a_1')t \Mod{r} \\
&\equiv \frac{r}{n}\left(\frac{n(n-1)}{2}-\frac{a_2t(a_2t-1)}{2}\right)m_2(a_1'-a_1)+ \frac{r}{n}(a_1-a_1')t \pmod{r}.
\end{aligned}
$}
\vspace{0.2cm}
\\
On the other hand by (\ref{matrixequa1}) we have 
\[
\begin{aligned}
(x_0'-x_t')-(x_0-x_t)&\equiv u_{12}(y_0-y_t)+(v_{34}-v_{3,t+4})\frac{r(r+1)}{2} \pmod{r} \\
&\equiv -u_{12}a_2t\frac{r}{n} \pmod{r},
\end{aligned}
\]
which follows since $v_{34}=-\frac{u_{23}}{n}$. Indeed, we have $y_0'=y_0$ and 
\[
\frac{r(r-n)}{2n}+a_2't\frac{r}{n}=y_t'=rv_{3,t+4}+y_t+u_{23}\frac{r}{n}=rv_{3,t+4}+\frac{r(r-n)}{2n}+a_2t\frac{r}{n}+u_{23}\frac{r}{n}
\]
hence $rv_{3,t+4}+u_{23}\frac{r}{n}=(a_2'-a_2)t\frac{r}{n}=0$ and $v_{3,t+4}=-\frac{u_{23}}{n}$.
\\

Combining the two expressions for $(x_0'-x_t')-(x_0-x_t)$ we get 
\[
\frac{r}{n}\left(\frac{n(n-1)}{2}-\frac{a_2t(a_2t-1)}{2}\right)m_2(a_1'-a_1)+ \frac{r}{n}(a_1-a_1')t+u_{12}a_2t\frac{r}{n}\equiv 0 \Mod{r}.
\]
Note that 
\[
\frac{r}{n}\left(\frac{n(n-1)}{2}\right)m_2(a_1-a_1')=r\left(\frac{n-1}{2}\right)m_2(a_1-a_1')\equiv 0 \mod r,
\]
since if $n$ is even then $2$ divides $a_1-a_1'$ and if $n$ is odd $2$ divides $n-1$. Thus we have
\begin{equation}\label{x1xt}
\left(\frac{r}{n}\frac{a_2t(a_2t-1)}{2}m_2+t\right)(a_1-a_1')+u_{12}a_2t\frac{r}{n}\equiv 0 \pmod{r}.
\end{equation}
By taking the sum of the expressions in (\ref{x1xt}) where we choose $t=1$ and $t=n-1$, we arrive at 
\[
\frac{r}{n}a_2^2m_2(a_1-a_1')\equiv 0 \pmod{r}.
\]
Then $\frac{r}{n}a_2^2m_2(a_1-a_1')=rk$ for some $k\in\mathbb{Z}$ hence $n$ divides $a_2^2m_2(a_1-a_1')$. Since $a_2$ and $m_2$ are both relatively prime to $n$, we conclude that  $a_1=a_1'$. 

For the other direction we will make use of \cite[Proposition 2.14]{errs} a number of times. Assume $m_i\equiv n_i \pmod{n}$, $i=1,2$ and  $\left(m_2^{-1}m_1-n_2^{-1}n_1\right)\frac{r(r-1)(r-2)}{3} \pmod{r}  \equiv 0 \pmod{r}$, then the entries, $y_t$, in the second row of the adjacency matrices are identical, and, it suffices to show for each $t$,
\[
\begin{aligned}
x_t'-x_t\equiv\left(m_2^{-1}m_1-n_2^{-1}n_1\right)\frac{r(r-1)(r-2)}{3n} +\sum_{l=1}^{r-2}l\frac{r}{n}(k_l'-k_l) \pmod{r}
\end{aligned}
\]
is an integer multiple of $\tfrac{r}{n}$. For the second term this is obvious. Additionally, it is obvious that this is also true for the first term, whenever $r$ is not a multiple of $3$. If $3\vert r$, then one finds that $3\vert m_2^{-1}m_1-n_2^{-1}n_1$, and the claim follows.

The adjacency matrices of each set of weights will then be identical after adding the third row to the first, and the first column to the $t$'th column in each an appropriate number of times.
\\
\\
(3). Proceeding as in part (1), we consider $C^*$-algebras coming from the weights $\underline{m}=(m_0,m_1,m_2,m_3)$ and $\underline{n}=(n_0,n_1,n_2,n_3)$ that are isomorphic. We denote by $x$ and $x'$ the elements coming from the top-right corner of the adjacency matrix as written in Lemma \ref{lemmam1} corresponding to $\underline{m}$ and $\underline{n}$ respectively. In a similar manner, it follows from \cite[Theorem 7.1]{errs}, and a computation that

\[\frac{a_1'r}{n}\equiv \frac{a_1r}{n} \pmod{r},\]
from which it follows that $m_1\equiv n_1 \pmod{n}$. Moreover, we obtain from the adjacency matrices that
\[x'-x\equiv \left( n_2^{-1}n_1-m_2^{-1}m_1 \right)\frac{r(2r-n)(r-n)}{6n^2}+ \left( n_1^{-1}-m_1^{-1} \right)\frac{r(r-n)(n-1)}{4n} \pmod{r},\]
and using \cite[Theorem 7.1]{errs}, we obtain $k_0,k_1\in \Zb$ such that
\[x'-x \equiv \frac{r(r+n)}{2n}k_0 +\frac{r}{n}k_1 \pmod{r}.\]
Consequently,
\vspace{0.2cm}
\\
\resizebox{\columnwidth}{!}{
$\left( n_2^{-1}n_1-m_2^{-1}m_1 \right)\frac{r(2r-n)(r-n)}{6n^2}+ \left( n_1^{-1}-m_1^{-1} \right)\frac{r(r-n)(n-1)}{4n}\equiv \frac{r(r+n)}{2n}k_0 +\frac{r}{n}k_1 \Mod{r}.$}
\vspace{0.1cm}
\\
Multiplying both sides by $2n$ yields
\[\left( n_2^{-1}n_1-m_2^{-1}m_1 \right)\frac{r(2r-n)(r-n)}{3n}+ \left( n_1^{-1}-m_1^{-1} \right)\frac{r(r-n)(n-1)}{2}\equiv 0 \Mod{r}.\]
Now, it is easy to check that the second term is always congruent to zero$\Mod{r}$, so we conclude that \[\left( n_2^{-1}n_1-m_2^{-1}m_1 \right)\frac{r(2r-n)(r-n)}{3n}\equiv 0 \pmod{r}.\]

Conversely, assuming that  
\begin{equation}\label{Eq:Main3}
    \left( n_2^{-1}n_1-m_2^{-1}m_1 \right)\frac{r(2r-n)(r-n)}{3n}=t \cdot r,
\end{equation}
we may argue as in the previous part, by showing that $x-x'$ is an integer multiple of $\frac{r}{n}$. It follows by a computation that
\[x-x'=\frac{r}{n}\left(\frac{2t+(m_1^{-1}-n_1^{-1})(r-n)(n-1)}{4} \right).\] Since $(m_1^{-1}-n_1^{-1})(r-n)(n-1)$ is necessarily divisible by 4 (recall that the only valid cases are the ones where $r$ and $n$ are both odd or both even, or $r$ is even and $n$ is odd, and that $m_1^{-1}$ and $n_1^{-1}$ are both odd if $r$ is even), it suffices to show that $t$ is even. If $r$ and $n$ are both odd, or both even, then this follows immediately by inspecting \eqref{Eq:Main3}. If $r$ is even and $n$ is odd, it follows that $n_2$ and $m_2$ are odd, and we may again conclude that $t$ is even.

It remains to be shown that 
$\left( n_2^{-1}n_1-m_2^{-1}m_1 \right)\frac{r(2r-n)(r-n)}{3n} \equiv 0 \pmod{r}$
if and only if $\left( n_2^{-1}n_1-m_2^{-1}m_1 \right)\frac{r(r-1)(r-2)}{3} \equiv 0 \pmod{r}$. It is routine to show that if $3$ does not divide $r$,, then $3$ divides either $2r-n$ or $r-n$. Consequently, if $3$ does not divide $r$, then the claim is trivial. If $\left(n_2^{-1}n_1-m_2^{-1}m_1 \right)\frac{r(r-1)(r-2)}{3} \equiv 0 \pmod{r}$ and $3\vert r$, then  $3\vert\left( n_2^{-1}n_1-m_2^{-1}m_1 \right)$ and one direction follows. The converse follows, remarking that $3\vert\left( n_2^{-1}n_1-m_2^{-1}m_1 \right)$ if $\left( n_2^{-1}n_1-m_2^{-1}m_1 \right)\frac{r(2r-n)(r-n)}{3n} \equiv 0 \pmod{r}$.


The proof of (2) is identical to that of (3) remarking that the adjacency matrix corresponding to the system of weights $\underline{m}\coloneqq (m_0,m_1,m_2,m_3)$ with $\gcd(m_1,r)=n$ and $\gcd(m_i,r)=1$ if $i\neq 2$ is the anti-transpose of the adjacency matrix corresponding to the system $\underline{m}'\coloneqq(m_0,m_2,m_1,m_3)$. By \cite[Definition 1.7]{GS07}, the adjacency matrices are related by the identity
\[A_{L_q^{7}(r;\underline{m})}= JA_{L_q^{7}(r;\underline{m}')}^TJ,\]
where $J$ is the involutory matrix whose entries are $1$ on the second diagonal and $0$ elsewhere.
\end{proof}

\begin{remark}
In this paper we have only dealt with the case there a single weight is coprime to the order of the acting group, $r$. There is however no clear reason why a similar result should be unobtainable in a more general setting. In particular, if we consider a list of weights $(m_0,m_1,m_2,m_3)$, then the methods for computing the adjacency matrices of the corresponding graph, could very likely be identical or similar, albeit more tedious, to the ones used above if at least one of $m_1$ or $m_2$ is coprime to $r$. If both weights are coprime, it is likely that an entirely different approach to counting is necessary since all methods employed so far have required one of them to be a unit of $\Zb_r$.
\end{remark}

\end{document}